\documentclass[10pt]{article}
\usepackage{amssymb}
\usepackage{amsmath}
\usepackage{theorem}
\usepackage{epsfig}
\usepackage{verbatim}
\usepackage{graphicx}
\usepackage{enumerate}
\usepackage{enumitem}
\usepackage{cancel}
\usepackage{mathtools}
\usepackage{hyperref}
\usepackage{datetime}

\usepackage{color}

\newlength{\hchng}
\newlength{\vchng}

\newcommand {\bc} {\begin{center}}
\newcommand {\ec} {\end{center}}

\theoremstyle{plain}

\newtheorem{thm}{Theorem}[section]

\newtheorem{defn}{Definition}[section]
\newtheorem{prop}[thm]{Proposition}

\newtheorem{rem}{Remark}[section]

\newtheorem{property}{Property}[section]
\newtheorem{claim}{Claim}
\newtheorem{claim1}{Claim}

\allowdisplaybreaks[2]

\newenvironment{proof}[1]{\begin{trivlist} \item[] {\em Proof of #1:}}{\hfill $\Box$
                      \end{trivlist}}

\newcommand{\ud}{\,\mathrm{d}}
\newcommand{\la}{\lambda}

\newcommand{\R}{\mathbb{R}}

\newcommand{\Om}{\Omega}

\newcommand{\pa}{\partial}
\newcommand{\eps}{\epsilon}
\newcommand{\vphi}{\varphi}

\title{Uniform level set estimates for ground state eigenfunctions}

\author{Thomas Beck}
\date{\today}                                           

\begin{document}
\maketitle
\begin{abstract}
We study the behaviour of the first eigenfunction of the Dirichlet Laplacian on a planar convex domain near its maximum. We show that the eccentricity and orientation of the superlevel sets of the eigenfunction stabilise as they approach the maximum, uniformly with respect to the eccentricity of the domain itself. This is achieved by obtaining quantitatively sharp second derivative estimates, which are consistent with the shape of the superlevel sets. In particular, we prove that the eigenfunction is concave (rather than merely log concave) in an entire superlevel set near its maximum. By estimating the mixed second and third derivatives partial derivatives of the eigenfunction, we also determine the rate at which a degree $2$ Taylor polynomial approximates the eigenfunction itself. 
\end{abstract}

\section{Introduction}

We are interested in studying properties of the first eigenfunction of the Dirichlet Laplacian on a planar convex domain $\Omega$, which are uniform with respect to the eccentricity of $\Omega$. In particular, we consider the behaviour of the level sets and derivatives of the eigenfunction as we approach its maximum. After a rotation and dilation, we may assume that the projection of $\Om$ onto the $y$-axis is the smallest projection of $\Om$ onto any direction, and we write $\Om$ as
\begin{align*}
\Om = \{(x,y) \in \R^2:x\in[a,b], f_1(x) \leq y \leq f_2(x) \}.
\end{align*}
Here $f_1(x)$  and $f_{2}(x)$ are convex and concave functions on the interval $[a,b]$ respectively, with $\min_{[a,b]}f_1(x) = 0$ and $\max_{[a,b]}f_2(x) = 1$. The first eigenvalue and eigenfunction of the Dirichlet Laplacian on $\Omega$ satisfies
\begin{eqnarray*}
    \left\{ \begin{array}{rlc}
    \Delta u(x,y) & = -\la u(x,y) & \text{in } \Om \\
   u(x,y) & = 0& \text{on } \pa \Om.
    \end{array} \right.
\end{eqnarray*}
where we normalise $u(x,y)$ to be positive inside $\Omega$ and attain a maximum of $1$. Our aim is then to study the eigenfunction and in particular its superlevel sets
\begin{align*}
\Om_{c} = \{(x,y) \in \Om: u(x,y) \geq c\}
\end{align*}
as $c$ approaches $1$. Since the eigenfunction $u(x,y)$ is log concave \cite{BL}, these superlevel sets  are convex subsets of $\Om$. Therefore, by the John lemma \cite{Jo}, for each $0\leq c \leq 1$, there exists a rectangle $R_{c}$ such that $R_{c}$ is contained with $\Om_{c}$ but a dilation of $R_{c}$ about its centre (with dilation factor bounded by an absolute constant) contains $\Om_{c}$. To determine the shape of the superlevel sets of $u(x,y)$ we want to study the length and orientations of the axes of the rectangles $R_{c}$. We in particular want to obtain estimates that are valid uniformly as the eccentricity (or equivalently the length of the interval $[a,b]$) becomes large. Instead of the diameter of $\Om$, the behaviour of the eigenfunction is governed by a different length scale $L$ determined by the geometry of the domain, first introduced in \cite{Je}.
\begin{defn} \label{defn:L}
Let a concave height function of the domain be given by $h(x) = f_2(x)-f_1(x)$. The length scale $L$ is then defined as the largest value such that
\begin{align*}
h(x) \geq 1-L^{-2} 
\end{align*}
for all $x\in I$, where $I$ is an interval of length at least $L$.
\end{defn}
Since the projection of $\Om$ onto the $y$-axis is the smallest, a geometric argument (Lemma 3.1, \cite{GJ}) ensures that $\max_{[a,b]}h(x)=1$. This fact and the concavity of $h(x)$ ensures that
\begin{align*}
(b-a)^{1/3} \leq L \leq b-a,
\end{align*}
but in general $L$ may be much smaller than $b-a$ (for example in the case of a right triangle, where $L \sim (b-a)^{1/3}$). Another way of generating a length scale of the same order of magnitude as $L$ is to define it to be the diameter of the inscribed rectangle in $\Om$ with smallest first Dirichlet eigenvalue (see Remark 2.2 in \cite{Je}).
\\
\\
In \cite{Je} an associated ordinary differential operator is introduced (see Section \ref{sec:strategy}), and the eigenvalues and eigenfunctions of this operator are used to generate appropriate test functions to provide bounds on the first eigenvalue $\la$ in terms of $L$, and to estimate the location and width of the nodal line of the second eigenfunction. In \cite{GJ1}, the authors give a sharper estimate on the nodal line, and in \cite{GJ} they study the location of the maximum of the first eigenfunction of $\Om$, and its behaviour near this maximum where they again use this approximate separation of variables to relate it to the first eigenfunction of an ordinary differential operator. As a straightforward consequence of their work, it is the length scale $L$ and orientation of the domain $\Om$ prescribed above, which determine the shape of the intermediate level sets of the eigenfunction $u(x,y)$.

\begin{thm}[Jerison, Grieser-Jerison] \label{thm:GJ}
Let $c_1>0$ be a small absolute constant, and consider $c$ with $c_1<c<1-c_{1}$. Then, we may take the rectangles $R_{c}$ to have axes lying along the coordinate directions, and with lengths comparable to $L$ and $1$ in the $x$ and $y$ directions respectively.
\end{thm}
The content of this theorem (and the results below) is in the high eccentricity case when $L$ is large. An important feature is that the dependence on $L$ is explicit, and that the estimates are uniform among all convex domains leading to the same value of $L$ (rather than an asymptotic analysis where a fixed convex domain is \textit{stretched} to obtain a sequence of domains of high eccentricity). One reason for this is that other than the convexity of $\Om$, no other regularity assumptions are placed on the domain and so the implied constants do not depend on the regularity of $\pa\Om$. In this theorem, the implied constants depend on $c_1$, and so it does not track the level set shape as $c$ tends to $1$. Our first aim is thus to study what happens as we approach the point where $u$ attains its maximum. 
\begin{thm} \label{thm:second}
For all $0<\eps\leq \tfrac{1}{2}$, we may take the rectangle $R_{1-\eps}$ to be centred at the maximum $(x^*,y^*)$ and have axes parallel to the coordinate directions. Moreover, denoting $I_{1-\eps}^{x}$, $I_{1-\eps}^{y}$ to be the projections of $R_{1-\eps}$ onto the $x$ and $y$ directions respectively, there exists an absolute constant $C>0$ such that
 \begin{align*}
 C^{-1}\sqrt{\eps}L \leq \left|I_{1-\eps}^{x}\right| \leq C\sqrt{\eps}L, \qquad C^{-1}\sqrt{\eps} \leq \left|I_{1-\eps}^{y}\right| \leq C\sqrt{\eps}.
 \end{align*}
\end{thm}
We will prove this theorem in Section \ref{sec:second}. As $\eps>0$ becomes small, the theorem establishes the shape of the level sets all the way to its maximum. It is important to note that the constant $C$ is universal and so can be taken independently of $\eps$, $L$  and the domain $\Om$ itself. Moreover, the estimates are uniform up to the $\Om_{1/2}$ superlevel set, and so the theorem shows that the eccentricity and orientation of the superlevel sets $\Om_{1-\eps}$ stabilise as $\eps$ tends to $0$, and that the level at which this occurs can be taken independently of $L$ and the domain $\Om$. The estimate in the theorem is a natural expectation for the behaviour of the level sets: Consider the polynomial
\begin{align*}
P(x,y) = 1 - \frac{(x-x^*)^2}{2L^2} - \frac{(y-y^*)^2}{2}
\end{align*}
attaining its maximum of $1$ at the $(x^*,y^*)$. Then, the level sets of $P(x,y)$ satisfy the property given in the statement of Theorem \ref{thm:second}, and we can informally view the theorem as saying that the eigenfunction $u(x,y)$ resembles the behaviour of the polynomial $P(x,y)$ as we approach the maximum. In particular, a key feature in the proof of Theorem \ref{thm:second} will be to establish sharp second derivative bounds for $u(x,y)$ near to its maximum, which are consistent with those attained by the polynomial $P(x,y)$. 
\\
\\
A classical result of Brascamp-Lieb \cite{BL} ensures that the eigenfunction is log concave in the whole of $\Om$. Alternative proofs of this result (and generalisations) have also been given in \cite{CF}, \cite{K}, \cite{KL}. We will in fact show that there exists $\eps_0>0$ sufficiently small (independently of $L$), that $u(x,y)$ is concave in the whole superlevel set $\Om_{1-\eps_0}$. We will also obtain quantitative lower bounds on its second derivatives in $\Om_{1-\eps_0}$ which are of the same order of magnitude as those of $P(x,y)$ (see Proposition \ref{prop:second}).
\\
\\
If we choose the degree $2$ polynomial more appropriately, we can obtain also a stronger approximation to the eigenfunction $u(x,y)$. Denote the degree $2$ Taylor polynomial of $u(x,y)$ at $(x^*,y^*)$ to be $P_2^*(x,y)$. That is,
\begin{align} \label{eqn:P2}
P_2^*(x,y) = 1 + \tfrac{1}{2} \pa_{x}^2u(x^*,y^*) (x-x^*)^2 +  \pa_{x}\pa_{y}u(x^*,y^*)(x-x^*)(y-y^*) +
  \tfrac{1}{2} \pa_{y}^2u(x^*,y^*) (y-y^*)^2.
\end{align} 
Then, our second main result establishes that the approximation of $u(x,y)$ by $P_2^*(x,y)$ improves uniformly as $(x,y)$ approaches the maximum. More precisely, we will show the following in Section \ref{sec:third}:
\begin{thm} \label{thm:third}
There exists an absolute constant $C>0$, such that for all $(x,y)\in\Omega_{1/2}$, we have the estimate
\begin{align*}
\left|u(x,y)-P_2^*(x,y)\right| \leq C\max\left\{L^{-3}|x-x^*|^3,|y-y^*|^3\right\}.
\end{align*}
\end{thm}
For those points $(x,y)$ away from the maximum, the conclusion of Theorem \ref{thm:third} follows immediately from Theorem \ref{thm:GJ}. To prove the theorem, we will therefore establish sharp bounds on the third derivatives of $u(x,y)$, which hold inside a whole superlevel set $\Omega_{1-\eps_0}$ (with $\eps_0>0$ a sufficiently small absolute constant, independent of $L$). However, to achieve this, the key quantity considered in the proof of the theorem is the mixed derivative
\begin{align*}
\pa_{x}\pa_{y}\left(\log u(x,y)\right).
\end{align*}
In the case where $u(x,y)$ is an exact separation of variables, this quantity is identically zero, and in the general case we will show that is $O(L^{-3})$ inside $\Omega_{1-\eps_0}$ (see Proposition \ref{prop:mixed}). An analogous bound on $\pa_{x}\pa_{y}u(x,y)$ itself does not hold, and so this estimate can be thought of as another indication that the eigenfunction resembles a product of two single variable functions. At the maximum $(x^*,y^*)$, the quantities $\pa_{x}\pa_{y}u(x^*,y^*)$ and $\pa_{x}\pa_{y}\left(\log u(x^*,y^*)\right)$ agree, and so the term $\pa_{x}\pa_{y}u(x^*,y^*)$ appearing in \eqref{eqn:P2} is in fact of size $O(L^{-3})$. This shows that the prescribed orientation of the domain $\Om$ lines up with the orientation of the Hessian $D^2u(x^*,y^*)$ at the maximum up to an $O(L^{-3})$ error. 
\\
\\
As well as the key work of Jerison and Grieser-Jerison  \cite{Je}, \cite{GJ1}, \cite{GJ}, we now recall some other work, which studies related problems, using similar techniques. The idea of approaching a two dimensional problem by using an approximate separation of variables and studying an associated ordinary differential operator has also been used extensively by Friedlander and Solomyak in \cite{FS1}, \cite{FS2}, \cite{FS3}. In these papers, they use this to obtain asymptotics for the eigenvalues, and the resolvent of the Dirichlet Laplacian. They use a semiclassical method by sending a small parameter $\eps$ to $0$ in order to give a one-parameter of \textit{narrow} domains, and then write asymptotics in terms of this small parameter. In Borisov-Freitas \cite{BF}, they use similar techniques to study the asymptotics of eigenfunctions and eigenvalues for a  class of planar, not necessarily convex, domains in the singular limit around a line segment. In Freitas and Krej\v ci\v r\'ik \cite{FK} they also relate the study of eigenfunctions and eigenvalues of a class of \textit{thin} two dimensional (not necessarily convex) domains, to an associated ordinary differential operator, and in particular use this to deduce properties of the nodal line of the second eigenfunction. 
\\
\\
The log concavity of the eigenfunction, and resulting convexity of its superlevel sets has been used previously in various situations. For example, in \cite{AC} moduli of convexity and concavity for the eigenfunctions of a class of Schr\"odinger operators are introduced. Under a convexity assumption on the potential, it is then possible to strengthen the log concavity of the first eigenfunction by finding an appropriate modulus of concavity, depending on the diameter of the domain. This allows them to estimate the spectral gap from below and prove the Fundamental Gap Conjecture. In Proposition \ref{prop:second} below, we also obtain a quantitative version of the log concavity near the maximum of the eigenfunction in terms of the length scale $L$, which provides a sharper estimate in the case where $L$ is much smaller than the diameter of the domain.   In \cite{GMS}, the definition of $L$ is generalised to a wider class of one dimensional Schr\"odinger operators, no longer necessarily requiring the convexity of the potential, and this is used to obtain lower bounds on the first eigenvalue.  In \cite{FJ} the convexity of the superlevel sets of the Green's function are used in a crucial way to prove third derivative estimates on the eigenfunction which are valid up to the boundary of the convex domain.
\\
\\
A natural question is to study to what extent Theorems \ref{thm:GJ}, \ref{thm:second}, \ref{thm:third} extend to three and higher dimensions. Some steps towards establishing Theorem \ref{thm:GJ} in three dimensions were given in \cite{Be1} and \cite{Be2}. In \cite{Be2} the analogous estimates for the shape of the intermediate level sets of the first eigenfunction were described for a class of two dimensional Schr\"odinger operators with convex potentials. In this case, there are two length scales $L_1$ and $L_2$ which govern the shape, and these length scales are determined by the eigenvalues of a family of associated ordinary differential operators. This was used in \cite{Be1} to prove the equivalent theorem to Theorem \ref{thm:GJ} for a class of \textit{flat} convex three dimensional domains such that the associated John parallelepiped has two large directions and one unit sized direction. 
\\
\\
In the rest of the paper we will proceed as follows. In Section \ref{sec:strategy} we will outline our strategy for the proof of Theorem \ref{thm:second}, and recall the properties of the eigenfunction from \cite{Je} and \cite{GJ} which we will need. In Section \ref{sec:second} we then prove Theorem \ref{thm:second} and also demonstrate the quantitative concavity of the eigenfunction in a whole level set near its maximum. Finally, in Section \ref{sec:third} we prove Theorem \ref{thm:third} by obtaining appropriate third derivative estimates for the eigenfunction. 
\\
\\
\textbf{Acknowledgements:} The author would like to thank David Jerison for the many useful conversations and suggestions regarding this paper and its exposition.

\section{Strategy and ingredients of the proof of Theorem \ref{thm:second}} \label{sec:strategy}

One approach to proving Theorem \ref{thm:second} is to apply a Harnack-type inequality to the second derivatives of $u(x,y)$ to show that they are of the same order of magnitude in a level set close to the maximum. One could then attempt to use the shape of the intermediate level sets in Theorem \ref{thm:GJ} to determine that the sizes of the second derivatives must match with those of $P(x,y) = 1 - \tfrac{(x-x^*)^2}{2L^2} - \tfrac{(y-y^*)^2}{2}$. However, to use such an approach there are two main issues that need to be overcome: To apply a Harnack-type inequality requires that the quantity is of one sign, and a priori, it is not clear that the second derivatives of $u(x,y)$ have a sign near to its maximum. Also, the desired eccentricities of the level sets of $u(x,y)$ are comparable to $L$, whereas the Harnack inequality applies to regions of bounded eccentricity. Since we want to sharply track the dependence on $L$, an overlapping ball type argument will not be sufficient.
\\
\\
We will get around these two issues as follows. In Proposition \ref{prop:max1} we will prove the estimates in Theorem \ref{thm:second} for an $L$-dependent range of $\eps>0$. In particular, we will establish them for $\eps \geq CL^{-3}$, which ensures that the remaining superlevel sets are contained within an $O(1)$ distance to the maximum of $u(x,y)$. This will allow us to reduce our analysis to a region which is suitable for applying a Harnack inequality. To overcome the lack of a sign for the second derivatives we will use the log concavity of $u(x,y)$ \cite{BL}. This gives the one-sided second derivative estimate of the form
\begin{align} \label{eqn:log-concave}
\pa_{\nu}^2u(x,y) \geq -\frac{\left(\pa_{\nu}u(x,y)\right)^2}{u(x,y)}
\end{align}
for every unit direction in $\nu$. In Proposition \ref{prop:max2} we will establish bounds on $\pa_{\nu}u(x,y)$ which are consistent with the desired shape of the level sets. Inserting this into \eqref{eqn:log-concave} will then allow us to apply a Harnack inequality to a quantity involving the second derivatives and complete the proof of Theorem \ref{thm:second}. 
\\
\\
In the rest of this section we introduce the associated ordinary differential operator and describe the relevant properties of the corresponding eigenfunction, which were established and follow from the work of Jerison and Grieser-Jerison \cite{Je}, \cite{GJ}. To motivate the definition of this operator, note that as $L$ increases, the height function $h(x)$ becomes more slowly varying. One can then think of the domain $\Om$ as resembling an approximate cartesian product of the interval $[a,b]$ and the intervals $[f_1(x),f_2(x)]$ of length $h(x)$. Considering this sort of approximation for the eigenfunction itself, since the first Dirichlet eigenvalue of $[f_1(x),f_2(x)]$ is equal to $\pi^2h(x)^{-2}$, a natural differential operator to consider in the $x$-direction is given by the following:
\begin{defn} \label{defn:ODE}
The ordinary differential operator $\mathcal{L}$ is given by
\begin{align*}
\mathcal{L} = -\frac{d^2}{dx^2} + V(x)
\end{align*}
with zero boundary conditions on the interval $[a,b]$. Here $V(x) = \pi^2h(x)^{-2}$ is a convex positive function, attaining a minimum of $\pi^2$.  Let $\mu$ be the first eigenvalue of $\mathcal{L}$ with corresponding eigenfunction $\vphi(x)$, normalised to be positive on $(a,b)$ with a maximum of $1$. 
\end{defn}
We can relate our two dimensional eigenfunction problem to this operator via an approximate separation of variables and a decomposition of the eigenfunction $u(x,y)$  into its first and higher Fourier modes in the $y$-direction. 
\begin{defn} \label{defn:Fourier}
Setting 
\begin{align*}
e(x,y) = \sqrt{\frac{2}{h(x)}}\sin \left(\pi\frac{y-f_1(x)}{h(x)}\right), \qquad \psi(x) =  \int_{f_1(x)}^{f_2(x)}u(x,y)e(x,y) \ud y,
\end{align*}
we decompose $u(x,y)$ as
\begin{align*}
u(x,y) = u_1(x,y) + u_2(x,y), \qquad \emph{ where } u_1(x,y) = \psi(x)e(x,y).
\end{align*}
\end{defn}

As the eccentricity of $\Omega$ increases, the height function $h(x)$ becomes slowly varying and one would expect that the behaviour of $\psi(x)$ should resemble that of $\vphi(x)$. This has indeed been confirmed in the work of  Jerison and Grieser-Jerison. In the proofs of Theorems \ref{thm:second} and \ref{thm:third}, we will use some of these estimates which relate $\la$, $u(x,y)$ to $\mu$, $\vphi(x)$, which for future reference we now list. 
\begin{property}[Eigenvalue bounds, Theorem A and Lemma 2.4(a) in \cite{Je}] \label{prop:eigenvalue}
There exists an absolute constant $C$ such that $\mu$ and $\la$ satisfy the eigenvalue bounds
\begin{align*}
\pi^2 + \frac{1}{CL^{2}} \leq \mu \leq \pi^2 + \frac{C}{L^{2}}, \qquad \pi^2 + \frac{1}{CL^{2}} \leq \la \leq \pi^2 + \frac{C}{L^{2}},
\end{align*}
and the sharper comparison estimate
\begin{align*} 
\mu \leq \la \leq \mu + CL^{-3}.
\end{align*}
\end{property}
This sharper eigenvalue estimate is established by building a test function out of $\vphi(x)$ and $e(x,y)$ in the variational formulation of $\la$.
\begin{property}[First derivative bounds, Lemma 2.4(b) in \cite{Je}, Lemma 4.3(a) in \cite{GJ}] \label{prop:first}
There exists an interval $J$ of length of length $C^{-1}L$, containing the maximum of $\vphi(x)$ and the projection of the maximum of $u(x,y)$ onto the $x$-axis, such that the first derivative bounds
\begin{align*} 
|\pa_{x}u(x,y)| \leq CL^{-1}, \qquad |\pa_{y}u(x,y)| \leq C, \qquad |\vphi'(x)| \leq CL^{-1}
\end{align*}
hold for all $(x,y)\in\Om$ with $x\in J$. 
\end{property}
These first derivative estimates are consistent with the shape of the intermediate level sets given in Theorem \ref{thm:GJ}, but are not sharp near the maximum of $u(x,y)$ and in particular not sufficient to alone establish Theorem \ref{thm:second}.
\begin{property}[First Fourier mode, (2.9) and (2.10) in \cite{GJ}] \label{prop:Fourier1}
The first Fourier mode $u_1(x,y) = \psi(x)e(x,y)$ satisfies
\begin{align*}
\left(\frac{d^2}{dx^2} - V(x)+ \mu\right)\psi(x) = \sigma(x),
\end{align*}
where $\int_{\tilde{J}}\left|\sigma(x)\right| \ud x \leq CL^{-3}$ for any interval $\tilde{J}\subset J$ of length $\leq1$. Moreover, $\psi(x)$ has a linear increase away from its maximum in the $x$-direction, which is consistent with the level set shape from Theorem \ref{thm:GJ}. More precisely,
\begin{align*}
|\psi'(x)| \geq C^{-1}L^{-2}|x| - CL^{-2}. 
\end{align*}
\end{property}
\begin{property}[Higher Fourier modes] \label{prop:Fourier2}
The eigenfunction $u(x,y)$ is well approximated by its first Fourier mode $u_1(x,y)$, since 
\begin{align*}
|u_2(x,y)| \leq CL^{-3}, \qquad  \left(\int_{W}\left|\nabla u_2(x,y)\right|^2 \ud x \ud y \right)^{1/2} \leq CL^{-3}
\end{align*}
for all $(x,y)\in\Omega$ with $x\in J$, and any unit sized disc $W$ with centre $(x_1,y_1)$, $x_1\in J$. 
\end{property}
\begin{rem} \label{rem:h-first}
By the concavity of the height function $h(x)$ and the definition of the length scale, we can ensure that on the interval $J$ the derivative bounds
\begin{align*}
\left|h'(x)\right| \leq CL^{-3}, \qquad \int_{\tilde{J}} \left|h''(x)\right| \ud x \leq CL^{-3},
\end{align*}
hold for all $x\in J$, intervals $\tilde{J} \subset J$. 
\end{rem}
\begin{proof}{Property \ref{prop:Fourier2}}
The pointwise bound on $u_2(x,y)$ is given by Theorem 2.1(a) in \cite{GJ}. To obtain the bound on $\nabla u_2(x,y)$, we first note that after a straightforward calculation, $u_2(x,y)$ satisfies the equation
\begin{align} \label{eqn:Fourier2a}
(\Delta +\la)u_2(x,y) = - F(x,y),
\end{align}
for $(x,y)\in\Om$. Here
\begin{align*}
F(x,y) = e(x,y)\sigma(x) + 2\pa_{x}e(x,y) \psi'(x) + \pa_{x}^{2}e(x,y)\psi(x),
\end{align*}
with $\sigma(x)$ as in Property \ref{prop:Fourier1}. In particular, since $h(x)$ is convex, with $|h'(x)|\leq CL^{-3}$ for $x\in J$, we have the bound
\begin{align} \label{eqn:Fourier2b}
\int_{\tilde{J}}|F(x,y)| \ud x \leq CL^{-3}
\end{align}
for any interval $\tilde{J} \subset J$ of length $\leq 1$. Now let $\chi(x,y)$ be a smooth cut-off function which is equal to $1$ inside $W$, equal to $0$ outside $2W$, the concentric double of $W$, such that it has two bounded derivatives. Then, integrating the equation in \eqref{eqn:Fourier2a} against $\chi(x,y)u_2(x,y)$, after an integration by parts, we obtain
\begin{align} \label{eqn:Fourier2c} \nonumber
 \int_{\Om}\chi(x,y)\left|\nabla u_2(x,y)\right|^2 \ud x \ud y = & \int_{\Om}-u_2(x,y)\nabla  \chi(x,y).\nabla u_2(x,y) + \la \chi(x,y)u_2(x,y)^2  \ud x \ud y \\ 
&  + \int_{\Om}F(x,y)\chi(x,y)u_2(x,y) \ud x \ud y .  
\end{align}
We can integrate by parts in the first term on the right hand to write it as
\begin{align*}
\frac{1}{2}\int_{\Om} u_2(x,y)^2 \Delta \chi(x,y) \ud x \ud y . 
\end{align*}
Thus, since $|u_2(x,y)| \leq CL^{-3}$, the first two terms on the right hand side of \eqref{eqn:Fourier2c} can be bounded by $CL^{-6}$. Finally, using $|u_2(x,y)|\leq CL^{-3}$ and the estimate in \eqref{eqn:Fourier2b} establishes the same bound for the final term in \eqref{eqn:Fourier2c}. Therefore, 
\begin{align*}
 \int_{\Om}\chi(x,y)\left|\nabla u_2(x,y)\right|^2\ud x \ud y   \leq CL^{-6},
\end{align*}
and since $\chi(x,y) = 1$ on $W$, this gives the desired estimate. 
\end{proof}
\begin{property}[Maximum location] \label{prop:max}
Assume that $\vphi(x)$ attains its maximum at $x=0$ and $u(x,y)$ attains its maximum of $1$ at $(x^*,y^*)$. Then, these maxima are close together and close to the thickest part of the domain in the sense that $|x^*|\leq C$, $|y^*-\tfrac{1}{2}|\leq CL^{-3/2}$ and 
\begin{align*}
V(x) - \mu = \frac{\pi^2}{h(x)^2} - \mu \leq - C^{-1}L^{-2}
\end{align*}
for all $x\in J$. 
\end{property}

\begin{proof}{Property \ref{prop:max}}
The estimates on $x^*$ and $V(x) - \mu$ have been established in Theorem 1.3 and Lemma 3.17 in \cite{GJ}. To estimate $|y^*-\tfrac{1}{2}|$, we first note that from Property \ref{prop:Fourier2} we have
\begin{align*}
\left|u(x^*,y) - \psi(x^*)e(x^*,y)\right| \leq CL^{-3}.
\end{align*}
Therefore, 
\begin{align*}
\left|1-\psi(x^*)e(x^*,y^*)\right| \leq CL^{-3} \qquad\text{and}\qquad \left|u\left(x^*,\tfrac{1}{2}f_1(x^*)+\tfrac{1}{2}f_2(x^*)\right) - \sqrt{2/h(x^*)}\psi(x^*)\right|\leq CL^{-3}  , 
\end{align*}
from which we obtain the estimate
\begin{align*}
1-CL^{-3}\leq \psi(x^*)e(x^*,y^*) =  \sqrt{\frac{2}{h(x^*)}}\psi(x^*) \sin \left(\pi\frac{y^*-f_1(x^*)}{h(x^*)}\right)\leq \left(1+CL^{-3}\right)\sin \left(\pi\frac{y^*-f_1(x^*)}{h(x^*)}\right) . \end{align*}
Since $h(x^*)\geq1-L^{-2}$, $0\leq f_1(x^*) \leq L^{-2}$, the desired estimate on $\left|y^*-\tfrac{1}{2}\right|$ then follows from the behaviour of the sine function around $\tfrac{\pi}{2}$. 
\end{proof}


\section{Level set estimates near the maximum} \label{sec:second}

In this section we will establish the desired properties of the level sets near to the maximum and prove Theorem \ref{thm:second}. As described at the beginning of Section \ref{sec:strategy}, the first step is to establish the level set shape for an $L$-dependent range of $\eps$.
\begin{prop} \label{prop:max1}
There exists an absolute constant $C_1$ such that Theorem \ref{thm:second} holds whenever $C_1L^{-3}\leq \eps\leq \tfrac{1}{2}$. That is, for this range of $\eps$, we may take the rectangle $R_{1-\eps}$ to have axes parallel to the coordinate directions. Denoting $I_{1-\eps}^{x}$, $I_{1-\eps}^{y}$ to be the projections of $R_{1-\eps}$ onto the $x$ and $y$ directions respectively, there exists an absolute constant $C>0$ such that
 \begin{align*}
 C^{-1}\sqrt{\eps}L \leq \left|I_{1-\eps}^{x}\right| \leq C\sqrt{\eps}L, \qquad C^{-1}\sqrt{\eps} \leq \left|I_{1-\eps}^{y}\right| \leq C\sqrt{\eps}.
 \end{align*}
 Moreover, we may take the interval $I_{1-\eps}^{y}$ to be centred at $y^*$.
 \end{prop}
We will now use this proposition to prove Theorem \ref{thm:second} and prove the proposition below in Section \ref{subsec:max1} using the properties of $u(x,y)$, $\vphi(x)$ and their relation given in the previous section. To deal with the remaining range of $\eps$ we will apply a Harnack inequality to a quantity involving the second derivatives of $u(x,y)$. In order to apply the Harnack inequality we need the relevant quantities to have a sign. We will ensure this by using the log concavity of $u(x,y)$, together with an estimate on its first derivatives.
\begin{prop} \label{prop:max2}
There exists an absolute constant $C_1$ such that whenever $\eps>0$ lies in the range $C_1L^{-2} \leq \eps \leq \tfrac{1}{2}$ the following estimate holds: There exists $C>0$ (independent of $\eps$) such that for all unit directions $\nu = (a,b)$, and all $(x,y)\in T_{\eps}$, the first derivatives of $u(x,y)$ satisfy
\begin{align*}
\left|\pa_{\nu}u(x,y)\right| = \left|\nu.\nabla u(x,y)\right| \leq C\sqrt{\eps}L^{-1}\max\{1,|b|L\}.
\end{align*}
Here $T_{\eps}$ is a rectangle, centred at $(x^*,y^*)$, with axes parallel to the coordinate directions, of side lengths comparable to $\sqrt{\eps}L$ and $\sqrt{\eps}$ in the $x$ and $y$ directions respectively. Moreover, for all $\eps>0$, we have the gradient estimate
\begin{align*}
\left|\nabla u(x,y)\right| \leq C\sqrt{\eps}
\end{align*}
for all $(x,y)\in\Om_{1-\eps}$. 
\end{prop}
We will prove Proposition \ref{prop:max2} in Section \ref{subsec:max2}. Note that for small $\eps>0$, the estimate on $\pa_{\nu}u(x,y)$ is sharper than the one given in Property \ref{prop:first}. It is also consistent with the level set shape given in Proposition \ref{prop:max1}. Assuming Propositions \ref{prop:max1} and \ref{prop:max2} for the time being we can now finish the proof of Theorem \ref{thm:second}:
\\
\\
Fix a small absolute constant $\eps_0>0$ to be specified below. For each unit direction $\nu$, let us define the function $F_{\nu}(x,y)$ by
\begin{align} \label{eqn:F-v}
F_{\nu}(x,y) =   C_3\eps_0\alpha_{\nu} - \pa_{\nu}^{2}u(x,y), 
\end{align}
where we have set $\alpha_{\nu} \coloneqq L^{-2}\left(\max\{1,|b|L\}\right)^2$. The constant $C_3$ is chosen sufficiently large (independently of $\eps_0>0$) using Proposition \ref{prop:max2} so that  $F_{\nu}(x,y)\geq0$ for all $(x,y)\in T_{\eps_0}$.  It also satisfies the equation
\begin{align} \label{eqn:Fv}
\left(\Delta + \la\right)F_{\nu}(x,y) = C_3\la\eps_0\alpha_{\nu} . 
\end{align}
We will apply the Harnack inequality to the function $F_{\nu}(x,y)$:
\begin{prop} \label{prop:Harnack}
Fix a small absolute constant $\eps_0>0$. Let $W$ be any convex subset of $\tfrac{1}{2}T_{\eps_0}$ with inner radius and diameter bounded above and below by $\sqrt{\eps_0}$ multiplied by an absolute constant. Here $\tfrac{1}{2}T_{\eps_0}$ corresponds to the set which is concentric to $T$ and half of its size. Then,
\begin{align*}
\sup_{W}F_{\nu}(x,y) \leq C_2\inf_{W}F_{\nu}(x,y) + C_2C_3\eps_0\alpha_{\nu},
\end{align*}
for some absolute constant $C_2$, independent of $\eps_0$, and the choice of $W$. 
\end{prop}
\begin{proof}{Proposition \ref{prop:Harnack}}
Given the set $W$, let $(x_0,y_0)$ be a point a distance comparable to $\sqrt{\eps_0}$ from the boundary  of $W$, and define $\tilde{F}_{\nu}(x,y)$ by
\begin{align} \label{eqn:tildeFv}
\tilde{F}_{\nu}(x,y) = F_{\nu}(\sqrt{\eps_0}(x-x_0),\sqrt{\eps_0}(y-y_0)).
\end{align}
Then, $\tilde{F}_{\nu}(x,y)$ is defined on the double of the convex set $\tilde{W}$ of inner radius and diameter comparable to $1$. Here $(x,y)\in\tilde{W}$ if and only if $(\sqrt{\eps_0}(x-x_0),\sqrt{\eps_0}(y-y_0))\in W$. From \eqref{eqn:Fv}, we see that it satisfies the equation 
\begin{align*}
\left(\Delta+\eps_0\la\right)\tilde{F}_{\nu}(x,y) = C_3\eps_0^2\la \alpha_{\nu}.
\end{align*}
 Since $\tilde{F}_{\nu}(x,y)\geq0$ we can therefore apply the Harnack inequality (see, for example, \cite{GT}, Theorems 8.17, 8.18) in order to get the bound
\begin{align*}
\sup_{\tilde{W}}\tilde{F}_{\nu}(x,y) \leq C_2\inf_{\tilde{W}}\tilde{F}_{\nu}(x,y) + C_2C_3\eps_0^2\alpha_{\nu}.
\end{align*}
Recalling the definition of $\tilde{F}_{\nu}(x,y)$ from \eqref{eqn:tildeFv} gives the desired bound on $F_{\nu}(x,y)$ in $W$.
\end{proof}
\begin{proof}{Theorem \ref{thm:second}}
We can now use Proposition \ref{prop:Harnack} to complete the proof of Theorem \ref{thm:second}. We first take the set $W = W_{\eps_0}$ to be equal to the ball of radius $\sqrt{\eps_0}$ centred at $(x^*,y^*)$. By Proposition \ref{prop:max1}, the superlevel sets $\Om_{1-\eps}$ resemble a rectangle of side lengths proportional to $\sqrt{\eps}L$ and $\sqrt{\eps}$, for $\eps\geq C_1L^{-3}$. Therefore, we have a value $\eps_1 = \eps_1(L)$ in this range of $\eps$ for which $\Om_{1-\eps_1}$ is contained within the set $W_{\eps_0}$. We will  show that for $\eps_0$ chosen sufficiently small, there exists a absolute constant $C$ such that for any unit direction $\nu = (a,b)$ we have the estimate
\begin{align} \label{eqn:second1}
 C^{-1}L^{-2}\left(\max\{|b|L,1\}\right)^2 \leq \inf_{W_{\eps_0}}-\pa_{\nu}^{2}u(x,y) \leq \sup_{W_{\eps_0}}-\pa_{\nu}^{2}u(x,y) \leq CL^{-2}\left(\max\{|b|L,1\}\right)^2 .
\end{align}
These second derivative bounds together with $u(x^*,y^*)=1$, $\nabla u(x^*,y^*)=0$ immediately give the desired shape of the superlevel sets inside $W_{\eps_0}$, and hence together with Proposition \ref{prop:max1} establish Theorem \ref{thm:second}.
\\
\\
For ease of notation, we will prove \eqref{eqn:second1} in the case $\nu=(1,0)$, but the same proof will work for all directions.  Suppose first that
\begin{align} \label{eqn:second2}
\sup_{W_{\eps_0}}-\pa_{x}^{2}u(x,y) \leq C_4^{-1}L^{-2},
\end{align}
for a large absolute constant $C_4$. Then, writing
\begin{align*}
1- u(x,y^*) = \int_{x}^{x^*}\pa_{\tau}u(\tau,y^*) \ud \tau =   \int_{x}^{x^*} \int_{x^*}^{\tau} \pa_{s}^2u(s,y^*) \ud s \ud \tau,
\end{align*}
we see that $1-u(x,y^*) \leq \tfrac{1}{2}\eps_1$ for all $|x-x^*|\leq \sqrt{C_4}\sqrt{\eps_1}L$. Thus, taking $C_4$ to be a sufficiently large absolute constant, this contradicts the shape of the superlevel set $\Om_{1-\eps_1}$. Hence \eqref{eqn:second2} fails for $C_4$ sufficiently large (independently of $\eps_0$). By Proposition \ref{prop:Harnack}, we obtain 
\begin{align} \label{eqn:second3} 
\sup_{W_{\eps_0}}F_{(1,0)}(x,y) \leq C_2\inf_{W_{\eps_0}}F_{(1,0)}(x,y) + C_2C_3\eps_0L^{-2} ,
\end{align}
with $F_{(1,0)}(x,y) = C_3\eps_0L^{-2} - \pa_{x}^{2}u(x,y)$. Thus, by taking $\eps_0>0$ sufficiently small, we have
\begin{align*}
-\pa_{x}^{2}u(x,y) \geq C^{-1}L^{-2}
\end{align*}
for all $(x,y)\in W_{\eps_0}$, for some absolute constant $C$.
\\
\\
Now suppose that 
\begin{align} \label{eqn:second2a}
\sup_{W_{\eps_0}}-\pa_{x}^{2}u(x,y) \geq C_4L^{-2},
\end{align}
for a large absolute constant $C_4$. Then, using \eqref{eqn:second3} again,  provided $C_4$ is sufficiently large, we have
\begin{align*}
-\pa_{x}^2u(x,y) \geq \tfrac{1}{2}C_4L^{-2}
\end{align*}
for all $(x,y)\in W_{\eps_0}$. But this again contradicts the shape of the superlevel set $\Om_{1-\eps_1}$, since by Proposition \ref{prop:max1}, $\Om_{1-\eps_1}$ contains the points $(x,y^*)$ for $x$ in an interval of length comparable to $\sqrt{\eps_1}L$ containing $x^*$.  Hence \eqref{eqn:second2a} fails for $C_4$ sufficiently large, which establishes \eqref{eqn:second1} for $\nu=(1,0)$ as required, and this completes the proof of Theorem \ref{thm:second}. 
\end{proof}

In the proof of Theorem \ref{thm:second} we established the expected second derivative bounds
\begin{align*}
C^{-1}L^{-2}\max\{1,|b|L\}^2 \leq -\pa_{v}^2u(x,y) \leq CL^{-2}\max\{1,|b|L\}^2 . 
\end{align*}
in a ball of radius $\sqrt{\eps_0}$ centred at the maximum of $u(x,y)$. However, the level set $\Omega_{1-\eps_0}$ has diameter comparable to  $\sqrt{\eps_0}L \gg 1$, and so we want to extend these bounds to a whole level set independently of the size of $L$. Since the diameter of the level set increases with $L$, we cannot use a Harnack chain-type condition to extend this bound. Instead we will use the fact that $u(x,y)$ is well-approximated by its first Fourier mode $u_1(x,y) = \psi(x) e(x,y)$ in a suitable sense (using Property \ref{prop:Fourier2}), and then obtain estimates on the derivatives of $u_1(x,y)$ itself (using Property \ref{prop:Fourier1}). 
\begin{prop} \label{prop:second}
There exist absolute constants $\eps^*, C^*>0$ such that for all unit directions  $\nu = (a,b)$, and all $(x,y)\in \Omega_{1-\eps^*}$, the second derivatives of $u(x,y)$ satisfy the upper and lower bounds
\begin{align*}
\frac{1}{C^*}L^{-2}\max\{1,|b|L\}^2 \leq -\pa_{v}^2u(x,y) \leq C^*L^{-2}\max\{1,|b|L\}^2 . 
\end{align*}
\end{prop}
\begin{proof}{Proposition \ref{prop:second}}
As in the statement of the Harnack inequality, Proposition \ref{prop:Harnack}, we fix a small absolute constant $\eps^*>0$ to be specified below, and work in the rectangle $\tfrac{1}{2}T_{\eps^*}$ centred at $(x^*,y^*)$ and of side lengths comparable to $\sqrt{\eps^*}L$ and $\sqrt{\eps^*}$. Given a unit direction $\nu = (a,b)$, define the function $G_{\nu}(x,y)$ by
\begin{align} \label{eqn:G-nu}
G_{\nu}(x,y) =  \pa_{\nu}u(x,y)-a\psi'(x)e(x,y) - b\psi(x)\pa_{y}e(x,y).
\end{align}
Since $\pa_{\nu}u_2(x,y) - G_{\nu}(x,y) = a\psi(x)\pa_{x}e(x,y)$, and $|\pa_{x}e(x,y)|\leq CL^{-3}$ for the range of $(x,y)$ under consideration, $G_{\nu}(x,y)$ satisfies the same $L^2$-bound as $\pa_{v}u_2(x,y)$ from Property \ref{prop:Fourier2}. That is,
\begin{align} \label{eqn:G-nu1}
 \left(\int_{W}G_{\nu}(x,y)^2 \ud x \ud y \right)^{1/2} \leq CL^{-3}.
 \end{align}
 We will therefore prove the proposition by first controlling $\pa_{\nu}G_{\nu}(x,y)$ in terms of $\pa_{\nu}^2u(x,y)$, and then using the  estimate in \eqref{eqn:G-nu1} to control  $\pa_{\nu}^2u(x,y)$.   
 \begin{claim} \label{claim:G1}
 For a given small absolute constant $\eps^*>0$, let $U_{\eps^*}$ be a subset of $\tfrac{1}{2}T_{\eps^*}$ of inner radius and diameter comparable to  $\sqrt{\eps^*}$. Then, there exists a constant $C>0$ (independent of $\eps^*$) such that, 
 \begin{align*}
 C^{-1}\eps^*\alpha_{\nu} \leq \int_{U_{\eps^*}} \pa_{\nu}G_{\nu}(x,y) - \pa^{2}_{\nu}u(x,y) \ud x \ud y  \leq C\eps^*\alpha_{\nu},
 \end{align*}
 where as before $\alpha_{\nu} = L^{-2}(\max\{|b|L,1\})^2$. 
 \end{claim}
 Assuming this claim for the time being, let us complete the proof of the proposition: Let $S$ be a square inside $\tfrac{1}{2}T_{\eps^*}$ with axes along the directions $\nu$ and its perpendicular $\nu^{\perp}$, and of side length comparable to $\sqrt{\eps^*}$. We will obtain an estimate on $\pa_{\nu}^{2}u(x,y)$ inside $S$, independent of the choice of $S$ within $\tfrac{1}{2}T_{\eps^*}$. Let us denote $S_1$, $S_2$, $S_3$ to be the three disjoint rectangles in $S$ of one third the width in the $\nu$-direction and the same height in the $\nu^{\perp}$-direction. Setting $(\cdot,\cdot)_{\nu}$ to be the coordinates of a point in the $\nu$ and $\nu^{\perp}$ basis, for any $(x_1,y_1) = (s_1,t)_{\nu}\in S_1$, $(x_3,y_3) = (s_3,t)_{\nu} \in S_3$, we have
\begin{align} \label{eqn:G-nu2}
\int G((s_3,t)_{\nu})-G((s_1,t)_{\nu})\ud t =\int \left(\int_{s_1}^{s_3}\pa_{v}G((s,t)_{\nu}) \ud s\right)\ud t,
\end{align}
with $s_3>s_1$ and the integrals in $t$ over the height of $S$ in the $\nu^{\perp}$ direction.  Suppose that
\begin{align*} 
\sup_{S}-\pa_{\nu}^{2}u(x,y) \geq C_1\alpha_{\nu},
\end{align*}
for a large absolute constant $C_1$. Then, applying the Harnack inequality from Proposition \ref{prop:Harnack} to $\eps^*$ and $S$, for $C_1$ sufficiently large, we have the lower bound $-\pa_{\nu}^2u(x,y) \geq \tfrac{1}{2}C_1\alpha_{\nu}$ in $S$. Combining this with the upper bound from Claim \ref{claim:G1} gives us an upper bound on the right hand side of \eqref{eqn:G-nu2}. We obtain
\begin{align*} 
\int_{S_{3}}G_{\nu}(x,y) \ud x \ud y-\int_{S_{1}}G_{\nu}(x,y) \ud x \ud y \leq -\tfrac{1}{4}C_1\left(\eps^*\right)^{3/2}\alpha_{\nu}.
\end{align*}
However, this is incompatible with the upper $L^2$-bounds from \eqref{eqn:G-nu1} with $W = S_1,S_3$, and this contradiction gives the desired upper bound on $-\pa_{\nu}^2u(x,y)$ in $S$. For $\eps^*>0$ sufficiently small, we can obtain the desired lower bound on $-\pa_{\nu}^2u(x,y)$ in $S$ in the analogous way, combining the Harnack inequality, the lower bound from Claim \ref{claim:G1} and \eqref{eqn:G-nu2}. Thus, to complete the proof of the proposition, it is left to prove Claim \ref{claim:G1}:
\\
\\
\noindent \textit{Proof of Claim 1:} Taking the derivative of $G_{\nu}(x,y)$ in the $\nu$ direction gives
\begin{align}\label{eqn:G1a} \nonumber
\pa_{v}G_{\nu}(x,y) & =  \pa^{2}_{\nu}u(x,y)  - a^2\psi''(x)e(x,y) +b^2\tfrac{\pi^2}{h(x)^2}\psi(x)e(x,y) - a^2\psi'(x)\pa_{x}e(x,y) \\
&  - 2ab\psi'(x)\pa_{y}e(x,y) - ab \psi(x)\pa_{y}\pa_{x}e(x,y) \\ \nonumber
& = \pa^{2}_{\nu}u(x,y) + G_{1,\nu}(x,y).
\end{align}
Here we have defined $G_{1,\nu}(x,y)$ to be all the remaining terms appearing on the right hand side of \eqref{eqn:G1a} except for $\pa^{2}_{\nu}u(x,y)$, and we need to estimate its integral over $U_{\eps^*}$. By integrating the equation for $\psi(x)$ from Property \ref{prop:Fourier1} we have 
\begin{align*}
C^{-1}|\tilde{J}|L^{-2} - CL^{-3} \leq\int_{\tilde{J}}-\psi''(x) \ud x \leq C|\tilde{J}|L^{-2} + CL^{-3}
\end{align*}
for any interval $\tilde{J}$ of length $\leq1$. Moreover, as we noted in Remark \ref{rem:h-first}, $|h'(x)| \leq CL^{-3}$, and so $|\pa_{x}e(x,y)| \leq CL^{-3}$ in $U_{\eps^*}$. Thus, there exists an absolute constant $C_1$ (independent of $\eps^*$) such that
\begin{align*}
C_1^{-1}\eps^*\alpha_{\nu} \leq \int_{U_{\eps^*}}  - a^2\psi''(x)e(x,y) +b^2\frac{\pi^2}{h(x)^2}\psi(x)e(x,y) \ud x \ud y \leq C_1\eps^*\alpha_{\nu}.
\end{align*}
To complete the proof of Claim \ref{claim:G1}, it is therefore sufficient to show that
\begin{align*}
\int_{U_{\eps^*}} \left|  - a^2\psi'(x)\pa_{x}e(x,y)- 2ab\psi'(x)\pa_{y}e(x,y) - ab \psi(x)\pa_{y}\pa_{x}e(x,y) \right| \leq \tfrac{1}{2}C_1^{-1}\eps^* \alpha_{\nu}
\end{align*}
for $\eps^*>0$ sufficiently small. But since the area of $U_{\eps^*}$ is comparable to $\eps^*$, this again follows from Property \ref{prop:Fourier1} and Remark \ref{rem:h-first}, which gives the estimates
\begin{align*}
 |\pa_{x}e(x,y)| \leq CL^{-3}, \qquad |\pa_{y}\pa_{x}e(x,y)| \leq CL^{-3}, \qquad |\psi'(x)| \leq C\sqrt{\eps^*}L^{-1}. 
\end{align*}
\end{proof}

\subsection{Proof of Proposition \ref{prop:max1}} \label{subsec:max1}

To prove this proposition, we will use the behaviour of the Fourier decomposition of $u(x,y)$ described in Properties \ref{prop:Fourier1} and \ref{prop:Fourier2}. In particular, it will be important to establish the approximate linear growth away from the point where $\psi(x)$ attains its maximum. Let $x=\bar{x}$ be this point. By the  lower bound on $|\psi'(x)|$ from Property \ref{prop:Fourier1}, we have $|\bar{x}| \leq C$, and so in particular it is in the middle half of the interval $J$. Since the equation in Property \ref{prop:Fourier1} gives
\begin{align*}
\psi'(x) = \int_{\bar{x}}^{x} (V(s) - \mu)\psi(s) + \sigma(s)  \ud s,
\end{align*}
the bounds on $\sigma(x)$ there, and the bounds on $V(x)-\mu$ from Property \ref{prop:max} imply that
\begin{align} \label{eqn:max1a}
 \frac{1}{CL^2}\left|x-\bar{x}\right| - \frac{C}{L^{3}} \leq |\psi'(x)| \leq \frac{C}{L^2}\left|x-\bar{x}\right| + \frac{C}{L^{3}}. 
\end{align}
Moreover, we have the estimate
\begin{align} \label{eqn:max1b}
1-\frac{C}{L^3} \leq \sqrt{\frac{2}{h(\bar{x})}}\psi(\bar{x}) \leq 1+\frac{C}{L^3},
\end{align}
since $\max_{y}e(\bar{x},y) = \sqrt{2/h(\bar{x})}$ for all $x$, and by Property \ref{prop:Fourier2} $|u_{2}(\bar{x},y)| \leq CL^{-3}$. Combining \eqref{eqn:max1a} and \eqref{eqn:max1b}, we see that for $\eps>C_1L^{-3}$, with $C_1$ sufficiently large, we have
\begin{align*}
 \sqrt{\frac{2}{h(x)}}\psi(x) \leq 1-2\eps &\qquad  \text{ for } \left|x-\bar{x}\right| \geq C_2\sqrt{\eps}L, \\
 \sqrt{\frac{2}{h(x)}}\psi(x) \geq 1-\tfrac{1}{2}\eps &\qquad \text{ for } \left|x-\bar{x}\right| \geq \frac{1}{C_2}\sqrt{\eps}L  ,
\end{align*}
for some large absolute constant $C_2$. Since for this range of $x$ we have
\begin{align*}
u(x,y) = \psi(x)e(x,y) + u_{2}(x,y) \leq \sqrt{\frac{2}{h(x)}}\psi(x) + CL^{-3}, \\
u(x,\tfrac{1}{2}) = \psi(x)e(x,\tfrac{1}{2}) + u_{2}(x,\tfrac{1}{2}) \geq \sqrt{\frac{2}{h(x)}}\psi(x)\left(1 -CL^{-3}\right) - CL^{-3},
\end{align*}
this means that the projection of the superlevel set $\Om_{1-\eps}$ onto the $x$-axis has length comparable to $\sqrt{\eps}L$. Moreover, for $y=\tfrac{1}{2}$, the superlevel set consists of an interval of length comparable to $\sqrt{\eps}L$ centred at $x=\bar{x}$.
\\
\\
It is straightforward to also show that for $\eps\geq C_1L^{-3}$ the projection of the superlevel set $\Om_{1-\eps}$ onto the $y$-axis has length comparable to $\sqrt{\eps}$, and also that for $x=\bar{x}$ the level set consists of an interval of length comparable to $\sqrt{\eps}$ centred at $y=\tfrac{1}{2}$: This follows from the behaviour of the sine function near $\pi/2$ since 
\begin{align*}
e(x,y) = \sqrt{\frac{2}{h(x)}} \sin \left(\pi \frac{y-f_1(x)}{h(x)} \right),
\end{align*}
and the estimate $|u_2(x,y)|\leq CL^{-3}$ from Property \ref{prop:Fourier2}. Thus, to complete the proof of Proposition \ref{prop:max1}, we need to show that the rectangle $R_{1-\eps}$ may be centred on the $y$-axis at $y^*$. But for this range of $\eps$, this is immediate, since from Property \ref{prop:max}, we have the estimate $|y^*-\tfrac{1}{2}|\leq CL^{-3/2}$. 
\begin{rem} \label{rem:max1}
Since both $|\bar{x}|\leq C$ and $|x^*|\leq C$, if we restrict to $C_1L^{-2}\leq \eps \leq \tfrac{1}{2}$ instead, we can ensure that the rectangle $R_{1-\eps}$ is centred at $(x^*,y^*)$. 
\end{rem}

\subsection{Proof of Proposition \ref{prop:max2}}\label{subsec:max2}

We will prove the first derivative bounds on $u(x,y)$ from Proposition \ref{prop:max2} in two steps: The first estimate is the one for $\pa_yu(x,y)$, which will follows from an inner radius estimate for the level set $\Om_{1-\eps}$  and elliptic regularity. To obtain the bound on $\pa_{x}u(x,y)$ we will use the convexity of $\Om_{1-\eps}$ and its shape from Proposition \ref{prop:max1} to get a boundary estimate and then use a comparison function and maximum principle to extend this to the interior. 
\begin{claim1} \label{claim:max2a}
 There exists an absolute constant $C>0$ such that 
\begin{align*}
|\nabla_{x,y}u(x,y)| \leq C\sqrt{\eps}
\end{align*}
for all $\eps>0$ and all $(x,y)\in \Om_{1-\eps}$. 
\end{claim1}
To prove this claim, we will use the generalised maximum principle, which we first recall.
\begin{prop} \label{prop:GMP}
Suppose that the functions $v_1$ and $v_2$ satisfy
\begin{align*}
\Delta v_1 + c(x)v_1 = 0,  \qquad \Delta v_2 + c(x)v_2 \leq 0,
\end{align*}
in a bounded domain $D$, where $c(x)$ is a bounded, continuous function. If in addition $v_1$ and $v_2$ are continuous in $\bar{D}$, $v_1>0$ in $D$ and $v_2>0$ in $\bar{D}$, with $v_1\leq v_2$ on $\pa D$, then
\begin{align*}
v_1 \leq v_2 \text{ in } D.
\end{align*}
\end{prop}
\begin{rem} \label{rem:GMP}
This is proven in \cite{PW}, Theorem 10, page 73, and follows from applying the usual maximum principle to the function $v_1/v_2$. 
\end{rem}
\begin{proof}{Claim \ref{claim:max2a}}
We first show that the superlevel set $\Om_{1-\eps}$ has inner radius at least $c\eps^{1/2}$, for a  small absolute constant $c>0$, which is independent of $\eps$. To do this, suppose that after a translation and rotation, $\Om_{1-\eps}$ is contained between the lines $y=\pm\alpha$, where $\alpha < c_1 \sqrt{\eps}$ for a small absolute constant $c_1$ to be chosen below. We use the comparison function
\begin{align*}
 \tilde{v}(x,y) \coloneqq \left(1-\tfrac{1}{2}\eps \right) \sin\left(\frac{\pi}{2} + \frac{\sqrt{\eps}y}{C_1\alpha}\right),
 \end{align*}
where $C_1>0$ is chosen (independently of $\alpha$ and $\eps$) so that $\tilde{v}(x,y) \geq 1-\eps$ for  all $(x,y)$ with $y = \pm\alpha$. Since $u(x,y) = 0$ on $\pa\Om$, denoting $\Om^{(\alpha)}$ to be the part of $\Om$ contained between $y=\pm\alpha$, this means that
\begin{align} \label{eqn:yloweruniform1}
u(x,y) \leq \tilde{v}(x,y)\qquad \text{for } (x,y) \in \pa\Om^{(\alpha)}.
\end{align}
 Moreover, the function $\tilde{v}(x,y)$ satisfies
\begin{align*}
\left(\Delta_{x,y} + \la \right)\tilde{v}(x,y) = - \left(\frac{\eps}{C_1^2\alpha^2}\right)\tilde{v}(x,y) + \la \tilde{v}(x,y).
\end{align*}
By Property \ref{prop:eigenvalue} we have $\la \leq \pi^2 +  CL^{-2}$, for some absolute constant $C>0$, and so
\begin{align} \label{eqn:yloweruniform2}
\left(\Delta_{x,y}  + \la\right) \tilde{v}(x,y) <\left(- \left(\frac{\eps}{C_1^2\alpha^2}\right) +\pi^2 + CL^{-2}\right)\tilde{v}(x,y) .
\end{align}
Provided $\alpha <c_1\sqrt{\eps}$, for $c_1$ sufficiently small (depending only on $C_1$ and $C$), we can ensure from \eqref{eqn:yloweruniform2} that
\begin{align} \label{eqn:yloweruniform3}
\left(\Delta_{x,y}  + \la\right)\tilde{v}(x,y) < 0.
\end{align}
Combining \eqref{eqn:yloweruniform1} and \eqref{eqn:yloweruniform3}, we see  by the generalised maximum principle in Proposition \ref{prop:GMP} that
\begin{align*}
u(x,y) \leq \tilde{v}(x,y) \qquad \text{for } (x,y)\in\Om^{(\alpha)}.
\end{align*}
However, $\tilde{v}(x,y) \leq 1-\frac{1}{2}\eps$ everywhere, whereas we know that $u(x,y)$ attains its maximum of $1$ for some $(x,y) \in\Om^{(\alpha)}$. Thus, $\alpha \geq c_1\sqrt{\eps}$ for some absolute constant $c_1$. Since $\Om_{1-\eps}$ is convex, this estimate (and the John lemma) is sufficient to establish the desired inner radius lower bound for $\Om_{1-\eps}$. 
\\
\\
For $\eps>0$ given, we now let $(x_1,y_1)\in \Om_{1-\eps}$ be a point a distance comparable to $\sqrt{\eps}$ from $\pa\Om_{1-\eps}$, and define the function $v(x,y)$ by
\begin{align*}
v(x,y) = \eps^{-1}\left(u\left(\sqrt{\eps}(x-x_1),\sqrt{\eps}(y-y_1)\right) - (1-\eps)\right).
\end{align*}
Then, this function satisfies
\begin{align*}
\Delta v(x,y) = -\eps\la v(x,y) - (1-\eps)\la,
\end{align*}
in the convex set $\tilde{\Om}_{1-\eps} = \{(x,y)\in\R^2: v(x,y) \geq 0\}$, and attains a maximum of $1$. By the above inner radius estimate, this set has inner radius bounded below by an absolute constant (uniformly in $\eps$), and so interior elliptic estimates show that the derivatives of $v(x,y)$ are bounded. Translating these bounds back to $u(x,y)$ proves the claim. 
\end{proof}

Since for the unit direction $\nu = (a,b)$, we have
\begin{align*}
\pa_{\nu}u(x,y) = a\pa_{x}u(x,y) + b\pa_{y}u(x,y),
\end{align*}
 to prove Proposition \ref{prop:max2} it is now sufficient to establish the following claim. 
\begin{claim1} \label{claim:max2b}
 There exist  absolute constants $C_1,C_2>0$ such that whenever $\eps$ is in the range $C_1L^{-2} \leq \eps \leq \tfrac{1}{2}$ the following holds: For all $(x,y)\in\Om_{1-\eps}$, such that $x$ is in an interval of length $C_2^{-1}\sqrt{\eps}L$, centred at $x^*$ we have the estimate
\begin{align*}
|\pa_{x}u(x,y)| \leq C_2\sqrt{\eps}L^{-1}.
\end{align*}
\end{claim1}
\begin{proof}{Claim \ref{claim:max2b}}
On the boundary of the superlevel set $\Om_{1-\eps}$, $u(x,y)$ is constant. Therefore, if we parameterise part of $\pa\Om_{1-\eps}$ by $y=g(x)$, we have
\begin{align} \label{eqn:x-deriv1}
(\pa_{x}u)(x,g(x)) + g'(x)(\pa_{y}u)(x,g(x)) = 0.
\end{align}
By Proposition \ref{prop:max1} and Remark \ref{prop:max1}, for this range of $\eps$, $\Om_{1-\eps}$ contains a rectangle centred at $(x^*,y^*)$ with axes along the coordinate directions of lengths comparable to $\sqrt{\eps}L$ and $\sqrt{\eps}$. Moreover, $\Om_{1-\eps}$ is convex. Combining this with \eqref{eqn:x-deriv1} ensures that there exists an interval $I_{\eps}$ centred at $x^*$ of length satisfying
\begin{align*}
C_2^{-1}\sqrt{\eps}L \leq \left|I_{\eps}\right| \leq C_2\sqrt{\eps}L
\end{align*}
such that
\begin{align} \label{eqn:x-deriv2}
|\pa_{x}u(x,y)| \leq C_2\sqrt{\eps}L^{-1}
\end{align}
for all $(x,y)\in\pa\Om_{1-\eps}$ with $x\in I_{\eps}$, and for some absolute constant $C_2$. Let $U_{\eps}$ equal the subset of $\Om_{1-\eps}$ such that $x\in I_{\eps}$. To complete the proof of the claim, we will extend the bound inside the level set, and show that the estimate in \eqref{eqn:x-deriv2} holds for all $(x,y)\in U_{\eps}$. By Proposition \ref{prop:max1}, the set $U_{\eps}$ is contained between the two horizontal lines $y = \tfrac{1}{2} \pm \tfrac{1}{4}C_3\sqrt{\eps}$, for some constant $C_3$ (independent of $\eps$). The boundary of $U_{\eps}$ consists of $4$ curves, where $\pa_{x}u(x,y)$ satisfies the bound from \eqref{eqn:x-deriv2} on the upper and lower boundaries, and by Property \ref{prop:first}
\begin{align*}
|\pa_{x}u(x,y)| \leq CL^{-1}
\end{align*}
on the part  of $\pa U_{\eps}$ consisting of the two vertical line segments. We now define a comparison function $W(x,y)$ by
\begin{align*}
W(x,y) = \frac{C_4\sqrt{\eps }}{L}\cosh \left(\frac{x-x^*}{\sqrt{\eps L}}\right)\sin \left(\frac{\pi}{2} + \frac{ \pi(y-\tfrac{1}{2})}{C_3\sqrt{\eps}}\right).
\end{align*}
Here $C_4>0$ is chosen sufficiently large so that
\begin{align*}
W(x,y) \geq C_2\sqrt{\eps} L^{-1}
\end{align*}
whenever $y\in[\tfrac{1}{2}-\tfrac{1}{4}C_3\sqrt{\eps},\tfrac{1}{2}+\tfrac{1}{4}C_3\sqrt{\eps}]$. Also, for this range of $y$, and $x\notin I_{\eps}$ we have
\begin{align*}
W(x,y) \geq \frac{C_4\sqrt{\eps}}{2L}\cosh(c\sqrt{L}) \geq CL^{-1}.
\end{align*}
Thus, in particular,
\begin{align} \label{eqn:x-deriv3}
W(x,y) \geq |\pa_{x}u(x,y)| \qquad \text{for } (x,y)\in \pa U_{\eps}.
\end{align}
The comparison function satisfies the equation
\begin{align} \label{eqn:x-deriv4}
(\Delta+\la)W(x,y) = \left(\frac{1}{\eps L} - \frac{\pi^2}{C_3^2\eps} + \la\right)W(x,y) ,
\end{align}
where the right hand side is negative for $(x,y)\in U_{\eps}$ and $\eps>0$ sufficiently small. Thus, combining \eqref{eqn:x-deriv3} and \eqref{eqn:x-deriv4}, we may apply the generalised maximum principle, Proposition \ref{prop:GMP}, to $W(x,y)$ and $\pa_{x}u(x,y)_{\pm}$ to conclude that
\begin{align*}
W(x,y) \geq \pa_{x}u(x,y)_{\pm}  \qquad \text{for } (x,y)\in  U_{\eps}.
\end{align*}
In particular, $|\pa_{x}u(x^*,y)|\leq C_4\sqrt{\eps}L^{-1}$ for $(x^*,y) \in \Om_{1-\eps}$. Since we can apply the same argument with $W(x,y)$ shifted by an amount comparable to $\sqrt{\eps}L$ in the $x$-direction, we have the desired bound on $\pa_{x}u(x,y)$, and this completes the proof of the claim.
\end{proof}

\section{Quadratic approximation to the eigenfunction} \label{sec:third}

In this section we will prove Theorem \ref{thm:third}, by showing that the eigenfunction $u(x,y)$ is sufficiently well approximated by its degree $2$ Taylor polynomial centred at its maximum $(x^*,y^*)$. As we remarked after the statement of the theorem, we only need to work inside a superlevel set $\Om_{1-\eps_0}$, where $\eps_0>0$ is a small absolute constant to be specified below. The key technical input will be the following proposition, which establishes bounds on the mixed partial derivatives of $u(x,y)$ near to $(x^*,y^*)$. 
\begin{prop} \label{prop:mixed}
There exists absolute constants $\eps_0, C_{0}>0$, such that for all $(x,y)\in\Omega_{1-\eps_0}$ the estimate
\begin{align*}
\left|\pa_{x}\pa_{y}\left(\log u(x,y)\right)\right| + \left|\nabla\pa_{x}\pa_{y}\left(\log u(x,y)\right)\right|  \leq C_{0}L^{-3}
\end{align*}
holds. 
\end{prop}
\begin{rem} \label{rem:mixed}
The estimates in Proposition \ref{prop:mixed} do not hold for the mixed derivatives of $u(x,y)$ itself. This  can be seen by
\begin{align*}
\pa_{x}\pa_{y}u(x,y) - u(x,y)\pa_{x}\pa_{y}\left(\log u(x,y)\right) = \frac{\pa_{x}u(x,y)\pa_{y}u(x,y)}{u(x,y)},
\end{align*}
since the right hand side will be $O(L^{-1})$ for some points in $\Om_{1-\eps_0}$. 
\end{rem}
\begin{rem} \label{rem:mixed2}
Since $\nabla u(x^*,y^*)=0$, the proposition does imply the estimate $|\pa_{x}\pa_{y}u(x^*,y^*)| \leq C_{0}L^{-3}$ at the maximum. This in particular means that the specified orientation of the domain $\Om$ (that the projection onto the $y$-axis is the smallest), agrees with the eigenvectors of the Hessian $D^2u(x^*,y^*)$ at the maximum, up to an error of size $O(L^{-3})$. 
\end{rem}
We will prove Proposition \ref{prop:mixed} in Section \ref{subsec:mixed} below and first use it to complete the proof of Theorem \ref{thm:third}: The third derivative estimates from Proposition \ref{prop:mixed} together with the estimates on the second derivatives of $u(x,y)$ from Proposition \ref{prop:second} immediately imply that
\begin{align*}
\left|\pa_{x}\pa_{y}^2u(x,y) \right| \leq C_{0}L^{-1}, \qquad \left|\pa_{x}^2\pa_{y}u(x,y) \right| \leq C_{0}L^{-2}
\end{align*}
in $\Om_{1-\eps_0}$. Thus, to prove Theorem \ref{thm:third} it is enough to show that 
\begin{align} \label{eqn:third1}
\left|\pa_{y}^3u(x,y) \right| \leq C_{0}, \qquad \left|\pa_{x}^3u(x,y) \right| \leq C_{0}L^{-3}
\end{align}
in the superlevel set $\Om_{1-\eps_0}$. The first estimate in \eqref{eqn:third1} follows from differentiating the eigenfunction equation in $y$. To obtain the estimate on $\pa_{x}^3u(x,y)$, we first see that
\begin{align*}
\pa_{x}^3\left(\log u(x,y)\right) + \pa_{x}\pa_{y}^2\left(\log u(x,y)\right) & = \pa_{x}\left(-\la +\left|\nabla \left(\log u(x,y)\right)\right|\right) \\
& = 2\pa_{x}^2\left(\log u(x,y)\right)\pa_{x}\left(\log u(x,y)\right) + 2\pa_{x}\pa_{y}\left(\log u(x,y)\right)\pa_{y}\left(\log u(x,y)\right). 
\end{align*}
Hence, by Propositions \ref{prop:second} and \ref{prop:mixed}, we have $|\pa_{x}^3(\log u(x,y))| \leq C_{0}L^{-3}$ in $\Omega_{1-\eps_0}$. Since
\begin{align*}
\pa_{x}^3u(x,y) - u(x,y)\pa_{x}^3\left(\log u(x,y)\right) = 3\frac{\pa_{x}^2u(x,y)\pa_{x}u(x,y)}{u(x,y)} -2  \frac{\left(\pa_{x}u(x,y)\right)^3}{u(x,y)^2} 
\end{align*}
we obtain the same bound on $\pa_{x}^3u(x,y)$ itself and this proves Theorem \ref{thm:third}.

\subsection{Proof of Proposition \ref{prop:mixed}} \label{subsec:mixed}

Let $W_{\eps_0} \subset \Om_{1-\eps_0}$ be a convex set with diameter and inner radius comparable to $\sqrt{\eps_0}$, where we fix $\eps_0>0$, with $\eps_0<\eps^*$ from Proposition \ref{prop:second}, so that the second derivative bounds hold for $u(x,y)$ in $\Om_{1-\eps_0}$. We will establish Proposition \ref{prop:mixed} by showing that the desired bounds hold in $W_{\eps_0}$ independently of its location within $\Om_{1-\eps_0}$. We set $v(x,y) = \pa_{x}\pa_y\left(\log u(x,y)\right)$, and we will use two crucial properties of this mixed derivative. First, the eigenfunction equation implies that $v(x,y)$ satisfies 
\begin{align} \label{eqn:mixed1}
\Delta v(x,y) + 2\nabla \log u(x,y).\nabla v(x,y) = 2\left(\la + \left|\nabla \log u(x,y)\right|^2\right)v(x,y). 
\end{align}
Second, setting 
\begin{align*}
v_1(x,y) = \pa_{x}\pa_{y}\left(\log u_1(x,y)\right) = \pa_{x}\pa_{y}\left(\log \psi(x) + \log e(x,y)\right) =  \pa_{x}\pa_{y}\left(\log e(x,y)\right),
\end{align*}
we have
\begin{align} \label{eqn:mixed2}
|v_1(x,y)| \leq  CL^{-3} \qquad \text{for } (x,y)\in W_{\eps_0}.
\end{align}
Since $\left|\nabla \log u(x,y)\right|$ is uniformly bounded, the equation in \eqref{eqn:mixed1} and $L^{p}$-elliptic regularity theory implies that, 
\begin{align} \label{eqn:mixed3}
\sup_{(x,y)\in \tfrac{1}{2}W_{\eps_0}} \left(|v(x,y)| + |\nabla v(x,y)|\right) \leq C_{0}\sup_{(x,y)\in W_{\eps_0}}|v(x,y)| . 
\end{align}
where $\tfrac{1}{2}W_{\eps_0}$ to the concentric half of $W_{\eps_0}$. Therefore, to prove Proposition \ref{prop:mixed} it is sufficient to bound $v(x,y)$ itself in $\Om_{1-\eps_0}$. We thus now pick $W_{\eps_0}$ so that it is centred at the point  where $|v(x,y)|$ attains its maximum in $\Om_{1-\eps_0}$, which we denote by $(x_1,y_1)$. Suppose that $|v(x_1,y_1)| \geq 2C_{1}L^{-3}$ for a large constant $C_{1}>0$ to be specified below. Then, if we set $v_2(x,y) = v(x,y) - v_1(x,y)$, by \eqref{eqn:mixed2} and \eqref{eqn:mixed3} we obtain the bound
\begin{align} \label{eqn:mixed4}
|v_2(x,y)| \geq C_{1}L^{-3}
\end{align}
for points $(x,y)$ in a square $B_{\eps_0}$ of side-length $c_{0}>0$ centred at $(x_1,y_1)$. Moreover, we may choose $c_{0}$ independently of the size of $C_{1}$.  We want to use \eqref{eqn:mixed4} to contradict the bounds on $\nabla u_2(x,y)$ from Property \ref{prop:Fourier2}. We first note that setting
\begin{align*}
w_2(x,y) = \log u(x,y) - \log u_1(x,y) = \log \left(1-\frac{u_2(x,y)}{u_1(x,y)}\right),
\end{align*}
we have $v_2(x,y) = \pa_{x}\pa_{y}w_2(x,y)$, and $\nabla w_2(x,y)$ satisfies the same bounds as $\nabla u_2(x,y)$ from Property \ref{prop:Fourier2}. In particular, this ensures that
\begin{align} \label{eqn:mixed5}
\left(\int_{B_{\eps_0}} \left|\nabla w_2(x,y)\right|^2\ud x \ud y\right)^{1/2} \leq CL^{-3}.
\end{align}
By the fundamental theorem of calculus, we can write 
$\pa_yw_2(x,y) = \pa_{y}w_2(\tilde{x},y) + \int_{\tilde{x}}^{x} v_2(s,y) \ud s,$
and so 
\begin{align*}
\int_{Y_{\eps_0}} \left(\pa_yw_2(x,y)\right)^2 \ud y \geq   \frac{1}{2}\int_{Y_{\eps_0}}  \left[\int_{\tilde{x}}^{x} v_2(s,y) \ud s\right]^2 \ud y -  \int_{Y_{\eps_0}}  \left(\pa_yw_2(\tilde{x},y)\right)^2 \ud y. 
\end{align*}
In the above $Y_{\eps_0}$ is the projection of the square $B_{\eps_0}$ onto the $y$-axis. Setting $\tilde{x} = x- \tfrac{1}{2}c_{0}$ and integrating in $x$ over the left half of $B_{\eps_0}$, implies that \eqref{eqn:mixed4} and \eqref{eqn:mixed5} are incompatible for $C_{1}>0$ sufficiently large. Thus, we obtain the desired bound on $|v(x,y)|$ in $\Om_{1-\eps_0}$ and this completes the proof of the proposition.

\end{document}